\documentclass[11pt]{article}
\usepackage{amsmath,amsthm,amssymb}
\usepackage[usenames,dvipsnames]{xcolor}
\usepackage{enumerate}
\usepackage{graphicx}
\usepackage{cite}
\usepackage{oands}
\usepackage{tikz}
\usepackage{changepage}
\usepackage{bbm}
\usepackage{mathtools}
\usepackage[margin=1in]{geometry}
\usepackage[pagewise,mathlines]{lineno}
\usepackage{appendix}
\usepackage{stmaryrd}
\usepackage{multicol}
\usepackage{microtype}
\usepackage[colorinlistoftodos]{todonotes}
\usepackage{dsfont}

\usepackage[pdftitle={Critical Liouville quantum gravity and CLE$_4$},
  pdfauthor={Morris Ang and Ewain Gwynne},
colorlinks=true,linkcolor=NavyBlue,urlcolor=RoyalBlue,citecolor=PineGreen,bookmarks=true,bookmarksopen=true,bookmarksopenlevel=2,unicode=true,linktocpage]{hyperref}

\theoremstyle{plain}
\newtheorem{thm}{Theorem}[section]

\newtheorem{lem}[thm]{Lemma}
\newtheorem{prop}[thm]{Proposition}

\def\@rst #1 #2other{#1}
\newcommand\MR[1]{\relax\ifhmode\unskip\spacefactor3000 \space\fi
  \MRhref{\expandafter\@rst #1 other}{#1}}
\newcommand{\MRhref}[2]{\href{http://www.ams.org/mathscinet-getitem?mr=#1}{MR#2}}

\theoremstyle{definition}
\newtheorem{defn}[thm]{Definition}
\newtheorem{remark}[thm]{Remark}

\numberwithin{equation}{section}

\newcommand{\dsb}{\begin{adjustwidth}{2.5em}{0pt}
\begin{footnotesize}}
\newcommand{\dse}{\end{footnotesize}
\end{adjustwidth}}

\newcommand{\ssb}{\begin{adjustwidth}{2.5em}{0pt}}
\newcommand{\sse}{\end{adjustwidth}}

\newcommand{\aryb}{\begin{eqnarray*}}
\newcommand{\arye}{\end{eqnarray*}}
\def\alb#1\ale{\begin{align*}#1\end{align*}}
\def\allb#1\alle{\begin{align}#1\end{align}}
\newcommand{\eqb}{\begin{equation}}
\newcommand{\eqe}{\end{equation}}
\newcommand{\eqbn}{\begin{equation*}}
\newcommand{\eqen}{\end{equation*}}

\newcommand{\BB}{\mathbbm}
\newcommand{\ol}{\overline}

\newcommand{\op}{\operatorname}

\newcommand{\re}{\operatorname{Re}}
\newcommand{\frk}{\mathfrak}
\newcommand{\eqD}{\overset{d}{=}}
\newcommand{\ep}{\varepsilon}

\newcommand{\wt}{\widetilde}
\newcommand{\wh}{\widehat}
\newcommand{\mcl}{\mathcal}

\newcommand{\bdy}{\partial}

\newcommand{\eps}{\varepsilon}

\newcommand{\cc}{{\mathbf{c}}}

\let\originalleft\left
\let\originalright\right
\renewcommand{\left}{\mathopen{}\mathclose\bgroup\originalleft}
\renewcommand{\right}{\aftergroup\egroup\originalright}

\title{Critical Liouville quantum gravity and CLE$_4$}
 \date{ }
 \author{
\begin{tabular}{c} Morris Ang\\[-5pt]\small UC San Diego \end{tabular}
\begin{tabular}{c} Ewain Gwynne\\[-5pt]\small University of Chicago \end{tabular}  
}

\begin{document}

\maketitle

\begin{abstract}
Consider a critical ($\gamma=2$) Liouville quantum gravity (LQG) disk together with an independent conformal loop ensemble (CLE) with parameter $\kappa=4$. We show that the critical LQG surfaces parametrized by the regions enclosed by the CLE$_4$ loops are conditionally independent critical LQG disks given the LQG lengths of the loops.  
We also show that the joint law of the LQG lengths of the loops is described in terms of the jumps of a certain $3/2$-stable process. 
Our proofs are via a limiting argument based on the analogous statements for $\gamma \in (\sqrt{8/3},2)$ and $\kappa = \gamma^2 \in (8/3,4)$ which were proven by Miller, Sheffield, and Werner (2020). 
Our results are used in the construction of a coupling of supercritical LQG with CLE$_4$ in another paper by the same authors.
\end{abstract}
  
%KEYWORDS: Liouville quantum gravity, Schramm-Loewner evolution, conformal loop ensemble, critical LQG
 
%AMS SUBJECT CLASS: 60J67 (SLE), 60D05 (geometric probability)

\tableofcontents

\bigskip
\noindent\textbf{Acknowledgments.}
We thank William Da Silva and Nina Holden for helpful discussions. We also thank two anonymous referees for helpful comments on an earlier version of this article.
M.A.\ was supported by the Simons Foundation as a Junior Fellow at the Simons Society of Fellows.
E.G.\ was partially supported by a Clay research fellowship. 
\medskip

\section{Introduction}
\label{sec-intro}

Liouville quantum gravity (LQG) with parameter $\gamma \in (0,2]$ is a family of canonical models of random ``surfaces" (two-dimensional Riemannian manifolds) which were first studied in the physics literature in the 1980s~\cite{polyakov-qg1,david-conformal-gauge,dk-qg}. The reason for the quotations is that LQG surfaces are too rough to be Riemannian manifolds in the literal sense. Heuristically, a $\gamma$-LQG surface parametrized by a domain $U\subset \BB C$ (or more generally a Riemannian manifold) is described by the random Riemannian metric tensor
\eqb \label{eqn-lqg-tensor}
g = e^{\gamma \Phi} (dx^2 + dy^2) 
\eqe  
where $dx^2 + dy^2$ is the Euclidean metric tensor and $\Phi$ is a random generalized function on $U$ which has similar local behavior to the Gaussian free field (GFF). 
We assume that the reader is familiar with the GFF; the unfamiliar reader can consult, e.g.,~\cite{shef-gff,pw-gff-notes,bp-lqg-notes}.  
More precisely, $\Phi$ should be sampled from the ``measure" $\exp\left( - \mcl S_{\op{L}}(\phi) \right) \,D\phi$ where $D\phi$ is the ``uniform measure on functions on $U$" and $\mcl S_{\op{L}}(\phi)$ is the Liouville action. See, e.g.,~\cite{dkrv-lqg-sphere,hrv-disk} for details. 

Although the Riemannian metric tensor~\eqref{eqn-lqg-tensor} does not make literal sense, one can define various objects associated with this metric tensor using regularization procedures. A few examples of objects that have been defined in this way include the LQG area measure~\cite{kahane,rhodes-vargas-log-kpz,shef-kpz}, LQG lengths of Schramm-Loewner evolution type curves~\cite{shef-zipper}, the LQG distance function~\cite{dddf-lfpp,gm-uniqueness}, and correlation functions of the exponential of $\Phi$~\cite{dkrv-lqg-sphere}. 

One of the most important and useful features of LQG is its connection to Schramm-Loewner evolution (SLE). SLE$_\kappa$ for $\kappa > 0$ is a one-parameter family of random fractal curves first introduced in~\cite{schramm0}. 
Roughly speaking, the connection between LQG and SLE takes the following form. Suppose we have a $\gamma$-LQG surface, described by a random generalized function $\Phi$, and independently from $\Phi$ we sample an SLE$_\kappa$ curve $\eta$ (or some variant thereof), with parameter $\kappa \in \{\gamma^2,16/\gamma^2\}$. Then, if the laws of $\Phi$ and $\eta$ are matched up in exactly the right way, the LQG surfaces parametrized by the complementary connected components of $\eta$ are independent, and their laws can be described explicitly. The first statement of this type was Sheffield's \textbf{quantum zipper} theorem~\cite{shef-zipper}. There are many  extensions of this result, corresponding to different variants of SLE and LQG, see, e.g.,~\cite{wedges,msw-simple-cle-lqg,ahs-welding}. Moreover, the relationships between SLE and LQG have a huge number of applications, e.g., to random planar maps, random permutations, fractal properties of SLE and LQG, and conformal field theory. See~\cite{ghs-mating-survey} for a survey of results about SLE and LQG and their applications.

In much of the literature about SLE and LQG, results are only proven in the subcritical case, i.e., $\gamma \in (0,2)$ and $\kappa \not=4$. However, many of these results can also be extended to the critical case $\gamma=2,\kappa=4$. 
For example, see~\cite{hp-welding} for a $\gamma=2$ analog of the quantum zipper result of~\cite{shef-zipper} and~\cite{ahps-critical-mating} for a $\gamma=2$ analog of mating of trees~\cite{wedges}.

The purpose of this note is to prove a $\gamma=2,\kappa=4$ analog of the relationships between LQG and conformal loop ensembles which originally appeared in~\cite{msw-simple-cle-lqg}. 
For $\kappa \in (8/3,4]$, the \textbf{(non-nested) conformal loop ensemble (CLE$_\kappa$)} on a proper simply connected open set $U\subset\BB C$ is a random countable collection of disjoint loops in $U$ which each locally look like SLE$_\kappa$ curves. The law of CLE$_\kappa$ is conformally invariant in the sense that if $\Gamma$ is a CLE$_\kappa$ on $U$ and $f : U\to V$ is a conformal map, then $f(\Gamma)$ is a CLE$_\kappa$ on $V$. CLE$_\kappa$ can be defined via branching SLE$_\kappa(\kappa-6)$ curves~\cite{shef-cle} or equivalently as the outer boundaries of the outermost clusters of a Brownian loop soup on $U$ with intensity $\cc(\kappa)/2$~\cite{shef-werner-cle}, where we define the \textbf{central charge}
\eqb \label{eqn-sle-cc}
\cc(\kappa) := 1-6\left(\frac{2}{\sqrt\kappa} - \frac{\sqrt\kappa}{2}\right)^2 \leq 1 .
\eqe

The particular LQG surface which interacts nicely with CLE$_\kappa$ is the \textbf{$\gamma=\sqrt\kappa$-LQG disk}, which was originally defined in~\cite{wedges,hrv-disk} and which is described by a particular random generalized function $\Phi$ on the unit disk $\BB D$. We recall one possible definition as Definition~\ref{def-lqg-disk} below. 
The main theorem of this paper is the following. 

\begin{thm} \label{thm-critical-cle4} 
	Let $\Phi$ be the random generalized function on $\BB D$ which describes the critical ($\gamma=2$) LQG disk with unit boundary length (Definition~\ref{def-lqg-disk}). 
	Let $\Gamma$ be a (non-nested) CLE$_4$ in $\BB D$ sampled independently from $\Phi$. 
	For each loop $\ell \in \Gamma$, let $U_\ell \subset \BB D$ be the open region enclosed by $\ell$. 
	If we condition on the $(\gamma=2)$-LQG lengths of all of the loops in $\Gamma$ (viewed as a countable decreasing sequence of real numbers), then the $(\gamma=2)$-LQG surfaces obtained by restricting $\Phi$ to the domains $\{U_\ell\}_{\ell\in\Gamma}$ (Definition~\ref{def-lqg-surface}) are conditionally independent critical LQG disks with the given boundary lengths. 
\end{thm}

The analog of Theorem~\ref{thm-critical-cle4} for $\gamma \in (\sqrt{8/3},2)$ and $\kappa=\gamma^2 \in (8/3,4)$ is immediate from~\cite[Theorem 1.1]{msw-simple-cle-lqg} (see Theorem~\ref{thm-subcritical-coupling} below). We will deduce Theorem~\ref{thm-critical-cle4} from this statement, plus a minor extension coming from~\cite{acsw-loop}, by taking the limit as $\gamma \to 2^-$. This is similar to the strategy used to establish relationships between critical LQG and SLE$_4$ in~\cite{hp-welding,ahps-critical-mating}. In fact, a key input in our proof is \cite[Lemma 4.5]{ahps-critical-mating}, which gives the convergence of the $\gamma$-LQG disk to critical LQG disk as $\gamma\to 2^-$. 

Theorem~\ref{thm-critical-cle4} can be viewed as a continuum analog of a certain Markovian property for random planar maps decorated by the $O(2)$ loop model. Indeed, for such random planar maps, if one conditions on the lengths of the outermost $O(2)$ loops, then the planar maps enclosed by the loops are conditionally independent planar maps with boundary decorated by the $O(2)$ loop model. See~\cite{bbg-recursive-approach,bbg-bending} for details. 

In addition to its intrinsic interest, Theorem~\ref{thm-critical-cle4} is an important input in the paper~\cite{ag-supercritical-cle4}, in which we establish a coupling between LQG and CLE$_4$ in the \textbf{supercritical} case, which corresponds to $\gamma\in \BB C$ with $|\gamma|=2$, or equivalently to central charge values in $(1,25)$. 

We will also establish an explicit description of the joint law of the 2-LQG lengths of the CLE$_4$ loops in the setting of Theorem~\ref{thm-critical-cle4}. The corresponding description in the subcritical case is stated at the end of \cite[Section 1]{ccm-cascade} and follows from results in~\cite{ccm-cascade,bbck-growth-frag,msw-simple-cle-lqg}  (see Theorem~\ref{thm-subcritical-lengths} below).  

\begin{thm}  \label{thm-critical-lengths} 
	Assume that we are in the setting of Theorem~\ref{thm-critical-cle4}. 
	Let $\zeta$ be a $3/2$-stable L\'evy process with no downward jumps, with L\'evy measure $\BB 1_{(x>0)} x^{-5/2} \,dx$.
	Let $P $ be the law of $\zeta$.  
	Let $\tau  = \inf\{t\geq 0 : \zeta(t) = -1\}$.
	Let $\{X^j\}_{j\in\BB N}$ be the sequence of upward jumps of $\zeta$ on $[0,\tau]$, listed in decreasing order of their sizes.
	Also let $\{\ell^j\}_{j\in\BB N}$ be the loops of the CLE$_4$ $\Gamma$, listed in decreasing order of their critical LQG lengths $\nu_\Phi(\ell^j)$ with respect to the critical LQG disk field $\Phi$. 
	Then the law of $\{\nu_\Phi(\ell^j)\}_{j\in\BB N}$ is the same as the law of $\{X^j\}_{j\in\BB N}$ under the re-weighted probability measure $(\tau^{-1} / \BB E[\tau^{-1}]) P $. 
\end{thm}

As explained in~\cite[Section 5.2.2]{ahps-critical-mating}, the joint law of $\{\nu_\Phi(\ell^j)\}_{j\in\BB N}$ can alternatively be described in terms of the growth fragmentation process of~\cite{ad-growth-fragmentation}. It was pointed out to us by W.\ Da Silva in private communication that Theorem~\ref{thm-critical-lengths} can be deduced from this description using growth fragmentation techniques. However, in this paper we will prove Theorem~\ref{thm-critical-lengths} via a limiting argument based on the analogous result for $\gamma \in (\sqrt{8/3},2)$.

\begin{remark}
	Kammerer~\cite[Theorem 1.1]{kammerer-critical-lqg-cle4} proves a stronger version of Theorem~\ref{thm-critical-cle4} which says that all of the domains cut out by an SLE$_4(-2)$ on a critical LQG disk are conditionally independent critical LQG disks given their boundary lengths, not just the domains that correspond to CLE$_4$ loops. This gives a $\gamma=2 , \kappa=4$ analog of~\cite[Theorem 1.2]{msw-simple-cle-lqg}, and is proven by taking a limit of SLE$_\kappa(\kappa-6)$ as $\kappa$ increases to $2$.
\end{remark}

\section{Preliminaries}
\label{sec-prelim}

\subsection{LQG surfaces}\label{sec-lqg-surfaces}

We define the infinite strip
\eqb  \label{eqn-strip}
\mcl S = \BB R \times (0,2\pi)
\eqe  
The \textbf{free boundary GFF} on $\mcl S$ is the centered Gaussian generalized function $ \wt\Phi$ on $\mcl S$ with covariance kernel
\eqb  \label{eqn-free-bdy}
G(z,w) = - \log |e^z- e^w| -\log |e^z- e^{\ol w}| + \max \{2 \re z , 0\} + \max\{2 \re w , 0\} .
\eqe 
One often views the free boundary GFF modulo additive constant.
Our choice of covariance kernel in~\eqref{eqn-free-bdy} corresponds to normalizing so that the average of $\wt\Phi$ over the vertical segment $\{0\} \times (0,2\pi)$ is zero. The free boundary GFF on simple connected domains other than $\mcl S$ can be defined via conformal invariance. We refer to~\cite[Section 6]{bp-lqg-notes} for more on the free boundary GFF. 

Following~\cite{shef-kpz,shef-zipper,wedges}, we view LQG surfaces as equivalence classes of domains decorated by generalized functions, modulo conformal maps. 

\begin{defn} \label{def-lqg-surface}
	Let $\gamma\in (0,2]$ and let 
	\eqb \label{eqn-Q-def}
	Q = Q_\gamma = \frac{2}{\gamma} + \frac{\gamma}{2} .  
	\eqe
	Let $k\geq 0$ and consider the set of $k+2$-tuples $(U,\Phi,x_1,\dots,x_k)$, where $U\subset \BB C$ is open, $\Phi$ is a generalized function on $U$ (which we will always take to be random, and in particular to be some variant of the GFF), and $x_1,\dots,x_k\in U\cup \bdy U$. We define an equivalence relation ${\sim}_Q$ on such $k+2$-tuples by
	\eqb \label{eqn-lqg-equiv}
	(U,\Phi,x_1,\dots,x_k) \sim_Q (\wt U,\wt\Phi, \wt x_1,\dots,\wt x_k)
	\eqe 
	if there is a conformal map $f : \wt U \to U$ such that 
	\eqb \label{eqn-lqg-coord}
	\wt\Phi = \Phi\circ f +Q\log|f'| \quad \text{and} \quad f(\wt x_j) = x_j,\quad\forall j=1,\dots,k .
	\eqe 
	A \textbf{$\gamma$-LQG surface} with $k$ marked points, is an equivalence class under ${\sim}_Q$.
	If $(U,\Phi,x_1,\dots,x_k)$ is an equivalence class representative of an LQG surface $\mcl S$, we refer to $\Phi$ as an \textbf{embedding} of $\mcl S$ into $(U,x_1,\dots,x_k)$. 
\end{defn}

We think of two equivalent $k+2$-tuples as in~\eqref{eqn-lqg-equiv} as two different parametrizations of the same LQG surface.
We can extend Definition~\ref{def-lqg-surface} in the obvious way to define LQG surfaces with additional decoration. 
We will need two particular special cases of this.
\begin{enumerate}[$(i)$]
	\item Let $U\subset \BB C$ be a domain with the topology of an annulus, viewed as a subset of the Riemann sphere, and let $\bdy^{\op{out}} U $ be one of the two connected components of $\bdy U$. A \textbf{$\gamma$-LQG surface with the annulus topology, with a distinguished boundary component}, is an equivalence class of triples $(U,\bdy^{\op{out}} U , \Phi)$, with two such triples $(U,\bdy^{\op{out}} U , \Phi)$ and $(\wt U,\bdy^{\op{out}} \wt U , \wt \Phi)$ declared to be equivalent (denoted by ${\sim}_Q$ as above) if there is a conformal map $f: \wt U \to U$ such that the condition~\eqref{eqn-lqg-coord} holds for $\Phi$ and also $f(\bdy^{\op{out}}\wt U) = \bdy^{\op{out}} U$. 
	\item Let $U\subset\BB C$ be open and let $G\subset  U$ be a connected set whose intersection with every compact subset of $U$ is compact. A \textbf{set-decorated $\gamma$-LQG surface} is an equivalence class of triples $(U,\Phi,G)$, with $U$ and $G$ as above, with two such triples $(U,\Phi,G)$ and $(\wt U , \wt \Phi , \wt G)$ declared to be equivalent (still denoted by ${\sim}_Q$) if there is a conformal map as in~\eqref{eqn-lqg-coord} which also satisfies $\Phi(\wt G) = G$. 
\end{enumerate}

Suppose now that $(U,\Phi,x_1,\dots,x_k) / {\sim}_Q$ is an LQG surface and $\Phi$ is a \textbf{GFF plus a continuous function}, i.e., $\Phi = \Phi^0 + \psi$ where $\Phi^0$ is a zero-boundary GFF on $U$ (or a whole-plane GFF, if $U=\BB C$) and $\psi$ is a (possibly random) continuous function on $U$, coupled with $\Phi^0$ in an arbitrary way. 
Using the theory of Gaussian multiplicative chaos, we can define on $U$ the \textbf{$\gamma$-LQG area measure} $\mu_\Phi = \mu_{\gamma,\Phi}$ which is the limit of regularized versions of $e^{\gamma \Phi(z)} \,d^2 z$, where $d^2 z$ denotes Lebesgue measure on $U$~\cite{kahane,shef-kpz,rhodes-vargas-review,berestycki-gmt-elementary} (see~\cite{shef-deriv-mart,shef-renormalization} for the critical case). 

If $\bdy U$ is a piecewise linear simple curve and $\Phi$ can be written as the sum of a free-boundary GFF on $U$ and a continuous function on $\ol U$, we can similarly define the \textbf{$\gamma$-LQG boundary length measure}  $\nu_\Phi = \nu_{\gamma,\Phi}$ which is the limit of regularized versions of $e^{(\gamma/2) \Phi(z)} |dz|$, where $|dz|$ denotes the length measure on $\bdy U$ (see, e.g.,~\cite[Section 6]{shef-kpz} or~\cite[Section 6.5]{bp-lqg-notes}). 

The measure $\mu_\Phi$ is compatible with the coordinate change formula~\eqref{eqn-lqg-coord}, in the following sense. Whenever $f , \Phi$, and $\wt\Phi$ are as in~\eqref{eqn-lqg-coord}, one has $\mu_\Phi = f_* \mu_{\wt\Phi}$~\cite{shef-kpz,shef-wang-lqg-coord} (see~\cite[Theorem 13]{shef-renormalization} for the critical case). 

If $\bdy U$ and $\bdy \wt U$ are both piecewise linear simple curves, then also $\nu_\Phi = f_*\nu_{\wt \Phi}$ (see~\cite[Section 6]{shef-kpz} for the subcritical case, the critical case can be obtained, e.g., by taking the limit as $\gamma\to 2^-$~\cite[Section 4.1]{aps-critical-lqg-lim}). 
Using this, one can extend the definition of $\nu_\Phi$ to LQG surfaces whose boundary is a simple curve (not necessarily piecewise linear) by conformally mapping to a domain whose boundary is a piecewise linear simple curve.
One can also extend the definition to domains whose boundaries have multiple connected components e.g., by using local absolute continuity.  
In particular, one can define the LQG length measure $\nu_\Phi$ on SLE$_\kappa$-type curves sampled independently from $\Phi$, for $\kappa = \gamma^2$~\cite{shef-zipper}. The measure $\nu_\Phi$ on such an SLE curve can also be defined as a Gaussian multiplicative chaos measure with respect to the Euclidean Minkowski content measure on the curve~\cite{benoist-lqg-chaos}. Similarly as with $\mu_\Phi$, the measure $\nu_\Phi$ is compatible with LQG coordinate change~\eqref{eqn-lqg-coord}, so it makes sense to speak of the $\gamma$-LQG boundary length measure of a $\gamma$-LQG surface.

\subsection{The LQG disk}
\label{sec-lqg-disk}

The main type of LQG surface we will consider in this paper is the $\gamma$-LQG disk for $\gamma\in (0,2]$, which was first defined in~\cite{wedges,hrv-disk}. We will give the definition from~\cite{wedges}. This definition is easier to state if we parametrize the surface by the infinite strip instead of the disk, due to the decomposition just below.

\begin{defn} \label{def-strip-decomp} 
	Let $\wt\Phi$ be a free-boundary GFF on the strip $\mcl S$, as defined in~\eqref{eqn-free-bdy}. 
	For $u \in \BB R$, we define $\wt\Phi^| : \mcl S \to \BB R$ to be the random function which, on each vertical segment $\{x\}\times (0,2\pi)$, is identically equal to the average of $\wt\Phi$ on this segment.
	We define the \textbf{lateral part} of $\wt\Phi$ by $\wt\Phi^\dagger = \wt\Phi - \wt\Phi^|$. 
\end{defn}

It follows from~\cite[Lemma 4.2]{wedges} that the random function $\wt\Phi^|$ (viewed modulo additive constant) and the random generalized function $\wt\Phi^\dagger$ are independent. We also note that $\wt\Phi^\dagger$ does not depend on the choice of additive constant for $\wt\Phi$.  

\begin{defn}[LQG disk] \label{def-lqg-disk}
	Let $\gamma \in (0,2]$. 
	Let $\mcl B :\BB R \to (-\infty,0]$ be a random function such that $\{\mcl B_s\}_{s\geq 0}$ and $\{\mcl B_{-s}\}_{s\geq 0}$ are independent and each has the law of the process $s\mapsto B_{2s} - (2/\gamma-\gamma/2)s$ conditioned to be negative for all time, where $B$ is a standard Brownian motion started from 0 (see \cite[Remark 4.4]{wedges} for an explanation of how to make sense of this conditioning).  
	With $\mcl S$ as in Definition~\ref{def-strip-decomp}, let $\Phi_0^|  : \mcl S \to \BB R$ be the random generalized function which is identically equal to $\mcl B_x$ on each vertical segment $\{x\}\times (0,2\pi)$. 
	Also let $\wt\Phi^\dagger$ be the lateral part of a free-boundary GFF on $\mcl S$, as in Definition~\ref{def-strip-decomp}, sampled independently from $\Phi_0^|$ and let 
	\eqbn
	\Phi_0 = \Phi_0^| + \wt\Phi^\dagger
	\eqen
	\begin{itemize}
		\item Let $\nu_{\Phi_0} = ``e^{\frac\gamma2\Phi_0(z)} \, |dz|"$ be the $\gamma$-LQG boundary length measure on $\partial \mcl S$ associated with $\Phi_0$ and let 
		\eqbn
		\Phi :=  \wh \Phi_0 -  \frac2\gamma\log \nu_{\wh\Phi_0}(\bdy \mcl S)   
		\eqen 
		where $\wh\Phi_0$ is sampled from the law of of $\Phi_0$, re-weighted by $[\nu_{\Phi_0}(\bdy\mcl S)]^{4/\gamma^2-1} / \BB E\left[[\nu_{\Phi_0}(\bdy\mcl S)]^{4/\gamma^2-1}  \right]$. 
		For $L > 0$, the \textbf{doubly marked $\gamma$-LQG disk with boundary length $L$} is the $\gamma$-LQG surface $(\mcl S , \Phi + \frac{2}{\gamma}\log L , -\infty,+\infty) / {\sim}_Q$. 
		\item The \textbf{singly marked  $\gamma$-LQG disk with boundary length $L$} is the  LQG surface $(\mcl S , \Phi + \frac{2}{\gamma}\log L , -\infty) / {\sim}_Q$, and the \textbf{unmarked $\gamma$-LQG disk with boundary length $L$} is the LQG surface $(\mcl S , \Phi + \frac{2}{\gamma}\log L) / {\sim}_Q$).
	\end{itemize}
\end{defn} 

Definition~\ref{def-lqg-disk} makes sense both for $\gamma\in (0,2)$ and for $\gamma=2$ since  $\BB E[[\nu_{\Phi_0}(\bdy\mcl S)]^{4/\gamma^2-1} ]$ is finite; when $\gamma < 2$ this moment is computed in \cite[Theorem 1.8, (1.34)]{rz-boundary-lcft} (with parameters $\beta = \gamma$ and $\mu_1=\mu_2=1$), and when $\gamma = 2$ it trivially evaluates to $1$. In the critical case $\gamma=2$, we have $2/\gamma-\gamma/2 = 0$, so each of $\{\mcl B_s\}_{s\geq 0}$ and $\{\mcl B_{-s} \}_{s\geq 0}$ has the law of $\sqrt 2$ times a standard Brownian motion conditioned to be negative, i.e., $(-\sqrt 2)$ times a 3d Bessel process started at 0.
Furthermore, we have $4/\gamma^2-1 = 1$, so $\wh\Phi_0 \eqD \Phi_0$.

For $\gamma < 2$, the $\gamma$-LQG disk was first introduced in \cite{wedges}, via a different description involving infinite measures. Their definition, after conditioning on the boundary length being $L$, agrees with Definition~\ref{def-lqg-disk}; see \cite[Remark 2.5]{hs-cardy-embedding} for an explanation of equivalence (it is stated for $\gamma=\sqrt{8/3}$, but the explanation holds for all $\gamma < 2$). We note for later use that the definition used in \cite{acsw-loop} agrees with that of \cite{wedges}, and hence Definition~\ref{def-lqg-disk}.
An alternative definition of the $\gamma$-LQG disk for $\gamma\in (0,2)$, based on conformal field theory, is given in~\cite{hrv-disk}. The definition there and the one in Definition~\ref{def-lqg-disk} are proven to be equivalent in~\cite{cercle-quantum-disk} (see also \cite{ahs-integrability}).

We will typically parametrize the LQG disks in this paper by the unit disk $\BB D$ rather than the strip $\mcl S$. An embedding of the LQG disk into $\BB D$ can be obtained from the embedding in Definition~\ref{def-lqg-disk} by applying a conformal map $\BB D \to \mcl S$ and changing coordinates via~\eqref{eqn-lqg-coord}.

\subsection{Topology on LQG surfaces} 
\label{sec-surface-topology}

We will need to talk about convergence of LQG surfaces, possibly with some decoration, as $\gamma \uparrow 2$. 
For this purpose we need to introduce a topology on the space of LQG surfaces.
The definitions in this section are similar to ones elsewhere in the literature, e.g.,~\cite[Section 2.2.5]{gwynne-miller-char}. 

\subsubsection{Set-decorated LQG surfaces with the annulus topology} 
\label{sec-annulus-metric}

Consider the set of triples $  (U , \bdy^{\op{out}} U , \mu ,  G)$ where 
\begin{itemize}
	\item $U\subset \BB C$ is an open domain with the topology of the annulus and $ \bdy^{\op{out}} U$ is a distinguished non-singleton boundary component of $U$, which we call the \textbf{outer boundary}. We allow the other boundary component of $U$ to be a single point, so, e.g., $U$ could be a punctured disk. 
	\item $\mu$ is a finite Borel measure on $U$.
	\item $G \subset U $ is a set whose intersection with each compact subset of $U$ is compact.
\end{itemize}
We declare that two such triples $(U , \bdy^{\op{out}} U , \mu ,  G)$ and $(\wt U , \bdy^{\op{out}} \wt U , \wt\mu ,  \wt G)$ are equivalent, denoted $(U , \bdy^{\op{out}} U , \mu ,  G) \sim (\wt U , \bdy^{\op{out}} \wt U , \wt\mu ,  \wt G)$, if there is a conformal map $f : U \to \wt U$ such that
\eqbn
f(\bdy^{\op{out}} U) = \bdy^{\op{out}} \wt U ,\quad f_*\mu = \wt\mu  ,  \quad \text{and} \quad f(G) = \wt G .
\eqen
Note that the existence of such a conformal map implies that the conformal moduli of the annuli $U$ and $\wt U$ are the same. 
Let $\BB A $ denote the set of equivalence classes under this equivalence relation. 

Now suppose that $Q  =2/\gamma + \gamma/2 > 2$ and $\mcl A  = (U , \bdy^{\op{out}} U ,  \Phi ,  G) /{\sim}_Q$ is a $\gamma$-LQG surface with the annulus topology decorated by a set $G$ (see the discussion just after Definition~\ref{def-lqg-surface}), where $U, \bdy^{\op{out}} U$, and $G$ satisfy the properties above and the $\gamma$-LQG area measure $\mu_\Phi$ a.s.\ has finite total mass. 
We identify $\mcl A$ with the equivalence class\footnote{\label{footnote-ep} The reason for scaling LQG areas by $\ep$ in~\eqref{eqn-surface-to-A} is that this is the scaling factor needed to make the subcritical LQG measure converge to the critical LQG measure as $\gamma\to 2^-$, see~\cite[Lemma 4.5]{ahps-critical-mating}.} 
\eqb \label{eqn-surface-to-A}
\frk A = (U , \bdy^{\op{out}} U ,  \ep^{-1} \mu_\Phi ,  G) / {\sim}  \in \BB A  ,\quad \text{where} \quad \ep = 2(2-\gamma) .
\eqe 
We make the same identification in the critical case $Q =\gamma=2$, except that we omit the factor of $\ep$ (i.e., we take $\ep =1$). It was shown in~\cite{bss-lqg-gff} that $\Phi$ is a.s.\ given by a measurable function of $\mu_\Phi$. Hence, a.s.\ $\mcl A$ is a measurable function of $\frk A$. 

An LQG surface $(U,\Phi,z)/{\sim}_Q$ with the disk topology, with exactly one marked interior point $z\in U$, can be identified with the set-decorated LQG surface with the annulus topology $(U \setminus\{z\} , \bdy U , \Phi , U\setminus \{0\} )/{\sim}_Q$, and can therefore also be identified with an element of $\BB A$. 

We define a metric $D_{\BB A}$ on $\BB A $ as follows. 
By the Riemann mapping theorem for annuli, for any $\frk A , \wt{\frk A} \in \BB A $ there exists unique numbers $r , \wt r \in [0,1)$ for which there exist equivalence class representatives 
\eqbn
( \BB D \setminus \ol{B_r(0)} , \bdy\BB D , \mu , G) \in \frk A \quad \text{and} \quad
( \BB D \setminus \ol{B_{\wt r}(0)} , \bdy\BB D , \wt\mu , \wt G) \in \wt{\frk A} .
\eqen
We view $\mu$ and $\wt\mu$ as measures on $\BB C$ (extended to be zero outside of their respective annuli) and we view $\ol G$ and $\ol{\wt G}$ as compact subsets of $\BB C$. We write $D_P$ and $D_H$ for the Prokhorov distance and Hausdorff distance, respectively, and we define 
\eqb \label{eqn-annulus-metric}
D_{\BB A}(\frk A , \wt{\frk A}) 
= \inf \left( D_P(\mu,\wt\mu) + D_H (\ol G, \ol{\wt G} ) + |r-\wt r| \right) 
\eqe 
where the inf is over all choices of equivalence class representatives as above (equivalently, over all rotations of $\BB D$). Then $D^{\op{disk}}$ is a metric on $\BB A$. 

Whenever we talk about convergence of LQG surfaces, it will be with respect to the topology induced by the above metric.

\subsubsection{LQG surfaces with the disk topology without a marked interior point} 
\label{sec-disk-metric} 

We will also have occasion to talk about LQG surfaces with the disk topology which do not have interior marked points (e.g., in the setting of Theorem~\ref{thm-critical-cle4}). In order to view such surfaces as random variables, we specify a topology (hence also a $\sigma$-algebra) on the set of such surfaces. We treat the case of LQG surfaces with no marked points. 
The case of one, two, or three marked points is treated similarly (except that we sample $3-k$ marked points from the boundary length measure instead of 3 at the end of the construction). 

Fix $Q > 0$ and consider the set of triples $ (U , \nu , \Phi)$ where  
\begin{itemize}
	\item $U\subset \BB C$ is an open simply connected domain, $U\not=\BB C$. 
	\item $\nu$ is a finite Borel measure on $\bdy U$ (viewed as a set of prime ends).
	\item $\Phi$ is a generalized function on $U$ which belongs to the local $-1$-order Sobolev space $\mcl H_{\op{loc}}^{-1}(U)$. 
\end{itemize}
We declare that two such triples $(U,\nu,\Phi)$ and $(\wt U , \wt\nu,\wt\Phi)$ are equivalent, denoted $(U,\nu,\Phi) \sim_Q (\wt U , \wt\nu,\wt\Phi)$, if there is a conformal map $f : \wt U  \to U $ such that
\eqbn
f_*\wt\nu =\nu \quad\text{and} \quad \Phi\circ f + Q \log |f'|  = \wt\Phi .
\eqen
We write $\BB S$ for the set of equivalence classes.

If $Q  =2/\gamma + \gamma/2 > 2$ and $\mcl S  = (U ,   \Phi  ) /{\sim}_Q$ is an unmarked $\gamma$-LQG surface with the disk topology whose $\gamma$-LQG boundary length measure $\nu_\Phi$ is well-defined and finite, we identify $\mcl S$ with the equivalence class 
\eqb \label{eqn-surface-to-S}
\frk S = (U ,   \nu_\Phi , \Phi ) / {\sim}_Q  \in \BB S .
\eqe  
Note that, unlike in~\eqref{eqn-surface-to-A}, we do not include a factor of $\ep^{-1}$ in~\eqref{eqn-surface-to-S} since the topology of this subsection is only used to define a $\sigma$-algebra, not to talk about convergence. 

We define a metric $D_{\BB S}$ on $\BB S $ as follows. Let $\frk S  = (U,\nu,\Phi)/{\sim}_Q$ and $\wt{\frk S} = (\wt U ,\wt \nu ,\wt\Phi )/{\sim}_Q$ be elements of $\BB S$. Let $x_1,x_2,x_3 \in \bdy U$ be independent samples from the probability measure $\nu /\nu(\bdy U)$ and let $f : \BB D \to U$ be the conformal map with $f(1) = x_1$, $f(i) = x_2$, and $f(-1) = x_3$. 
Define
\eqbn
\nu^\bullet := f^* \nu \quad \text{and} \quad   \Phi^\bullet := \Phi\circ f + Q\log|f'|  .
\eqen
Similarly define $\wt\nu^\bullet$ and $\wt\Phi^\bullet$. 

We define $D_{\BB S}(\frk S , \frk S')$ to be the Prokhorov distance between the laws of the pairs $(\nu^\bullet,\Phi^\bullet)$ and $(\wt\nu^\bullet,\wt\Phi^\bullet)$ with respect to the Prokhorov distance for measures on $\bdy\BB D$ in the first coordinate and the $\mcl H^{-1}_{\op{loc}}$ metric on the second coordinate. We emphasize that the only randomness here is coming from the choices of $x_1,x_2,x_3$ and $\wt x_1,\wt x_2,\wt x_3$. Then $D_{\BB S}$ is a metric on $\BB S$.

\section{Proof of the main results}
\label{sec-critical-cle4}

\subsection{Outline of the proof}
\label{sec-outline}

From \cite[Theorem 1.1]{msw-simple-cle-lqg}, we have the following analog of Theorem~\ref{thm-critical-cle4} for $\gamma \in (\sqrt{8/3},2)$ and $\kappa = \gamma^2 \in (8/3,2)$.

\begin{thm} \label{thm-subcritical-coupling} 
	Let $\gamma \in (\sqrt{8/3},2)$ and let $(\BB D , \Phi_\gamma)/{\sim}_{Q_\gamma}$ be an unmarked $\gamma$-LQG disk with unit boundary length, with $\Phi_\gamma$ as in Definition~\ref{def-lqg-disk}. 
	Let $\Gamma_\gamma$ be a (non-nested) CLE$_{\kappa=\gamma^2}$ in $\BB D$ sampled independently from $\Phi_\gamma$. 
	For each loop $\ell \in \Gamma_\gamma$, let $U_\ell \subset \BB D$ be the open region enclosed by $\ell$. 
	If we condition on the $\gamma$-LQG lengths of all of the loops in $\Gamma_\gamma$, then the $\gamma $-LQG surfaces obtained by restricting $\Phi_\gamma$ to the domains $\{U_\ell\}_{\ell\in\Gamma_\gamma}$ (Definition~\ref{def-lqg-surface}) are conditionally independent $\gamma$-LQG disks with the given boundary lengths.
\end{thm}

In Section~\ref{sec-details}, we will deduce Theorem~\ref{thm-critical-cle4} from Theorem~\ref{thm-subcritical-coupling} by taking a limit as $\gamma\to 2^-$. 
The key facts which allow us to do this are the convergence in law of the $\gamma$-LQG disk to the critical LQG disk as $\gamma\to 2^-$~\cite{ahps-critical-mating} (see Lemma~\ref{lem-disk-conv}); and the convergence in law of CLE$_\kappa$ to CLE$_4$ as $\kappa\to 4^-$, which follows from the Brownian loop soup description of CLE$_4$ (see Lemma~\ref{lem-cle-conv}). 

These two convergence statements do not immediately imply Theorem~\ref{thm-critical-cle4} for two reasons. First, it is not obvious that the $\gamma$-LQG lengths of the CLE$_{\gamma^2}$ loops converge to the 2-LQG lengths of the CLE$_4$ loops (although it is easy to see that the loop lengths, appropriately re-scaled, are tight, see Lemma~\ref{lem-length-tight}). 
Second, conditional distributions do not in general behave nicely under convergence in distribution. 

For the proof of Theorem~\ref{thm-critical-cle4}, we will assume that the embeddings of our $\gamma$-LQG disks for $\gamma \in (\sqrt{8/3},2]$ are chosen so that the origin corresponds to a point sampled uniformly from the $\gamma$-LQG area measure.
We also let $\{\ell_\gamma^j\}_{j\in\BB N}$ be the sequence of CLE$_{\gamma^2}$ loops, in decreasing order of their $\gamma$-LQG lengths. We let $J_\gamma$ be the index of the origin-surrounding loop. 

Fix $k\in\BB N$ and let $\mcl S_2^k$ denote the (critical) LQG surface parameterized by the region enclosed by $k$th CLE$_4$ loop $\ell_2^k$. We need to show that the conditional law of $\mcl S_2^k$ given the 2-LQG length $\nu_2(\ell_2^k)$ and the LQG surfaces $\{\mcl S_2^j\}_{j \in \BB N\setminus \{k\}}$ is that of a critical LQG disk with the given boundary length.

The main idea in our proof is to study the conditional law of $\mcl S_2^k$ given $\nu_2(\ell_2^k)$, $\{\mcl S_2^j\}_{j \in \BB N\setminus \{k\}}$, and the event $\{J_2 = k\}$ that $\ell_2^k$ is the origin-surrounding loop. 
We will give two different descriptions of this conditional law. The first comes from Theorem~\ref{thm-subcritical-coupling}, \cite[Proposition 5.4]{acsw-loop}, and a limiting argument (Lemma~\ref{lem-weighting}), and is in terms of a Radon-Nikodym derivative with respect to the law of a critical LQG disk. 
The second comes from the fact that the origin is a point sampled from the LQG area measure (Lemma~\ref{lem-cond-weight}) and is in terms of a Radon-Nikodym derivative with respect to the conditional law of $\mcl S_2^k$ given only $\nu_2(\ell_2^k)$ and $\{\mcl S_2^j\}_{j \in \BB N\setminus \{k\}}$ (i.e., not conditioning on $\{J_2 = k\}$). Equating these two descriptions will show that the conditional law of $\mcl S_2^k$ given $\nu_2(\ell_2^k)$ and $\{\mcl S_2^j\}_{j \in \BB N\setminus \{k\}}$ is that of a critical LQG disk with the given boundary length, as required. This holds for any $k$, which concludes the proof of Theorem~\ref{thm-critical-cle4}.  

The proof of Theorem~\ref{thm-critical-lengths}, given in Section~\ref{sec-length-proof}, is also based on a limiting argument started from the analogous statement in the subcritical case, see Theorem~\ref{thm-subcritical-lengths} below. 
To deduce Theorem~\ref{thm-critical-lengths} from this statement, our main task is to show that as $\gamma\to 2^-$, the joint law of the $\gamma$-LQG lengths of the CLE$_{\gamma^2}$-loops $\{\ell_\gamma^j\}_{j\in\BB N}$ converges to the analogous joint law for $\gamma=2$ (Proposition~\ref{prop-ordered-loop-conv}). 

To do this, for each $\gamma \in (\sqrt{8/3},2]$ we will consider countably many marked points $\{Z_\gamma(m)\}_{m\in\BB N}$ sampled from the $\gamma$-LQG area measure. We define $J_\gamma(m)$ to be the index of the CLE$_{\gamma^2}$ loop which surrounds $Z_\gamma(m)$. The results of Section~\ref{sec-details} (in particular, Lemma~\ref{lem-weighting}) allow us to show that as $\gamma\to 2^-$, the joint law of the $\gamma$-LQG lengths of the loops $\ell_\gamma^{J_\gamma(m)}$ converge to the analogous joint law for $\gamma=2$ (Lemma~\ref{lem-infinite-marked-pts}). The convergence of the joint law of the LQG lengths of $\{\ell_\gamma^j\}_{j\in\BB N}$ will follow from this together with some tightness statements (Lemma~\ref{lem-index-tight} and~\ref{lem-inverse-tight}).

\subsection{Proof of Theorem~\ref{thm-critical-cle4}}
\label{sec-details}

For $\gamma \in (\sqrt{8/3} , 2)$, let
\eqb \label{eqn-rescaled-measures}
\ep = \ep_\gamma := 2(2-\gamma) ;
\eqe   
c.f.\ Footnote~\ref{footnote-ep}.
Also let $\Phi_\gamma$ be an embedding into $\BB D$ of a $\gamma$-LQG disk with boundary length $\ep$, chosen so that 0 is a point sampled from the $\gamma$-LQG area measure and 1 is a point sampled from the $\gamma$-LQG length measure. More precisely, let $\Phi_\gamma^0$ be an arbitrary embedding into $\BB D$ of a $\gamma$-LQG disk of unit boundary length (obtained, e.g., by starting with an embedding of the LQG disk in $\mcl S$ as in Definition~\ref{def-lqg-disk}, conformally mapping $\mcl S$ to $\BB D$, and changing coordinates via~\eqref{eqn-lqg-coord}).
Conditional on $\Phi_\gamma^0$, let $Z \in \BB D$ and $W \in \bdy\BB D$ be sampled from its associated $\gamma$-LQG area measure and $\gamma$-LQG length measure, respectively, and take $Z$ and $W$ to be conditionally independent given $\Phi_\gamma^0$. Let $f : \BB D \to \BB D$ be the conformal automorphism with $f(0) = Z$ and $f(1) = W$, and let 
\eqbn
\Phi_\gamma :=  \Phi_\gamma^0 \circ f  + Q_\gamma \log|f'|  + \frac{2}{\gamma} \log \ep ,\quad \text{where} \quad  Q_\gamma := 2/\gamma + \gamma/2 .
\eqen 
For $\gamma  = 2$, instead let $\Phi_2$ be the embedding into $\BB D$ of a critical LQG disk of unit boundary length (instead of boundary length $\ep$), defined in the same manner as $\Phi_\gamma$ above.

For $\gamma\in (\sqrt{8/3},2]$, we write $\nu_\gamma$ and $\mu_\gamma$ for the $\gamma$-LQG length measure and $\gamma$-LQG area measure associated with $\Phi_\gamma$.  
Also let $\Gamma_\gamma$ be a CLE$_{\kappa=\gamma^2}$ on $\BB D$, sampled independently from $\Phi_\gamma$.  

\begin{remark}\label{remark-moment}
	Even though the origin is a point sampled from the $\gamma$-LQG area measure, we do \emph{not} weight the law of our LQG disks by the total mass of the $\gamma$-LQG area measure (i.e., we are not looking at quantum disks with a marked interior point in the sense of, e.g., \cite{wedges}). One reason why this is important is that in the critical case we have $\BB E[\mu_2(\BB D)] =\infty$, so weighting by the total area for $\gamma=2$ does not make sense. Indeed, for $\gamma < 2$ the law of $\mu_\gamma(\BB D)$ is the inverse gamma distribution with shape $\frac4{\gamma^2}$ and scale $\eps^2/(4 \sin (\frac{\pi\gamma^2}4))$ by \cite[Theorem 1.2]{ag-disk} and \cite[Theorem 1.3]{ars-fzz}, so by Lemma~\ref{lem-disk-conv} below the law of $\mu_2(\BB D)$ is the inverse gamma distribution with shape $1$ and scale $(2\pi)^{-1}$. Thus $\BB E[\mu_\gamma(\BB D)] < \infty$ for $\gamma <2$ and $\BB E[\mu_2(\BB D)] = \infty$.
\end{remark}

A crucial input in our proof is the following convergence statement.

\begin{lem}[{\cite{ahps-critical-mating}}] \label{lem-disk-conv}
	With $\ep$ as in~\eqref{eqn-rescaled-measures}, we have the joint convergence in law $(\Phi_\gamma , \ep^{-1} \mu_\gamma , \ep^{-1} \nu_\gamma) \to (\Phi_2 , \mu_2 , \nu_2)$ as $\gamma \to 2^-$, where the first coordinate is given the topology of the Sobolev space $\mcl H^{-1}(\BB D)$ and the second two coordinates are given the topology of weak convergence of measures on $\BB D$ and $\bdy\BB D$, respectively.
\end{lem}
\begin{proof}
	This is immediate from \cite[Lemma 4.5]{ahps-critical-mating} (see also~\cite[Remark 4.9]{ahps-critical-mating}). 
	Note that in \cite{ahps-critical-mating}, the authors re-define the LQG area and LQG boundary length measures to be $\ep^{-1}$ times the definitions used here, whereas we do not do this re-definition here. 
\end{proof}

For $\gamma\in (\sqrt{8/3} ,2]$, we write $\Gamma_\gamma = \{\ell_\gamma^{j}\}_{j\in\BB N}$, where the loops are enumerated in decreasing order of their $\nu_\gamma$-lengths, i.e., 
\eqb \label{eqn-ordered-loops}
\nu_\gamma(\ell_\gamma^1) > \nu_\gamma(\ell_\gamma^2)  > \dots .
\eqe 
We abbreviate
\eqbn
U_\gamma^j = U_{\ell_\gamma^j} := \left(\text{Open region enclosed by $\ell_\gamma^j$} \right)  
\eqen
and we define the $\gamma$-LQG surface
\eqb  \label{eqn-loop-surface}
\mcl S_\gamma^j := \left(U_\gamma^j , \Phi_\gamma|_{U_\gamma^j} \right) / {\sim}_{Q_\gamma} .
\eqe 

Let $J_\gamma \in \BB N$ be chosen so that
\eqb \label{eqn-marked-loop-def}
0 \in U_\gamma^{J_\gamma} .
\eqe 
Since our embedding is chosen so that 0 corresponds to a point sampled from the probability measure $\mu_\gamma / \mu_\gamma(\BB D)$, we see that the conditional law of $J_\gamma$ given the LQG surfaces $\{\mcl S_\gamma^j \}_{j\in\BB N}$ is 
\eqb \label{eqn-marked-loop-law} 
\BB P\left[ J_\gamma = k \,|\, \{\mcl S_\gamma^j \}_{j\in\BB N} \right] 
= \frac{\mu_\gamma(U_\gamma^k)}{\mu_\gamma(\BB D)}  .
\eqe 
We need the following tightness result.

\begin{lem} \label{lem-length-tight}
	The origin-surrounding loop lengths $\{\ep^{-1} \nu_\gamma(\ell_\gamma^{J_\gamma}) \}_{\gamma \in (\sqrt{8/3},2)}$ are tight.
\end{lem}
\begin{proof}
	Since $\nu_\gamma(\ell_\gamma^{J_\gamma}) \leq \nu_\gamma(\ell_\gamma^1)$ by definition, it suffices to show that $\ep^{-1} \nu_\gamma(\ell_\gamma^1)$ is tight.
	
	From Theorem~\ref{thm-subcritical-coupling}, we know that the conditional law of the LQG surface $\mcl S_\gamma^1$ given $\nu_\gamma(\ell_\gamma^1)$ is that of a $\gamma$-LQG disk with the given boundary length. Therefore, the conditional law given $\nu_\gamma(\ell_\gamma^1)$ of 
	\eqbn
	\frac{ \ep^2 }{[\nu_\gamma(\ell_\gamma^1)]^2}  \mu_\gamma(U_\gamma^1) 
	\eqen
	is the same as the law of the total $\gamma$-LQG area of a $\gamma$-LQG disk with boundary length $\ep$. By Lemma~\ref{lem-disk-conv}, this total area re-scaled by $\ep^{-1}$ converges in law to an a.s.\ positive random variable, so for every $p\in (0,1)$ there exists $C_p > 0$, depending only on $p$, such that for every $\gamma \in (\sqrt{8/3},2)$,  
	\eqb \label{eqn-big-area-pos0}
	\BB P\left[ \ep  \mu_\gamma(U_\gamma^1)  \geq  C_p  [\nu_\gamma(\ell_\gamma^1)]^2  \right] \geq p .
	\eqe 
	That is,
	\eqb \label{eqn-big-area-pos}
	\BB P\left[ \ep^{-1} \mu_\gamma(U_\gamma^1)  \geq  C_p \ep^{-2} [\nu_\gamma(\ell_\gamma^1)]^2   \right] \geq p .
	\eqe 
	We have $\ep^{-1}\mu_\gamma(U_\gamma^1) \leq \ep^{-1} \mu_\gamma (\BB D ) $, which is a tight random variable by Lemma~\ref{lem-disk-conv}. This combined with~\eqref{eqn-big-area-pos} shows that the random variable $\ep^{-2} [\nu_\gamma(\ell_\gamma^1)]^2 $, and hence also the random variable $\ep^{-1} \nu_\gamma(\ell_\gamma^1) $, is tight. 
\end{proof}

\begin{figure}[t]
	\begin{center}
		\includegraphics[scale = 0.75]{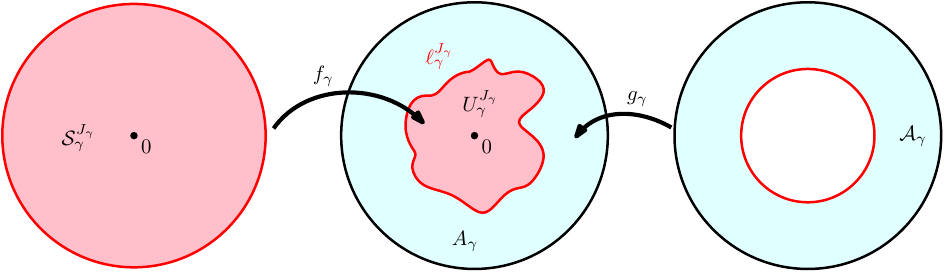}  
		\caption{\label{fig-in-out} Illustration of the domains and LQG surfaces defined in \eqref{eqn-marked-loop-def}, \eqref{eqn-outer-annulus}, and the surrounding discussion. The LQG surface $\mcl A_\gamma$ is additionally decorated by the CLE gasket $G_\gamma$. 
		}
	\end{center}
\end{figure}

Let
\eqbn
A_\gamma := \BB D\setminus \ol U_\gamma^{J_\gamma} 
\eqen
be the annular region outside of the origin-containing loop of $\Gamma_\gamma$. 
Also define the intersection of the CLE gasket with $A_\gamma$ by
\eqbn
G_\gamma :=  \ol{ \bigcup_{\ell\in \Gamma_\gamma } \ell } \cap A_\gamma  .
\eqen
We note that $\Gamma_\gamma$ is a.s.\ determined by $G_\gamma$ (the loops of $\Gamma_\gamma$ are the boundaries of the complementary connected components of $\ol G_\gamma$). We define the set-decorated LQG surface, with the annulus topology and the outer boundary distinguished, 
\eqb \label{eqn-outer-annulus}
\mcl A_\gamma := \left( A_\gamma , \bdy\BB D ,  \Phi_\gamma|_{A_\gamma} , G_\gamma \right) / {\sim}_{Q_\gamma} ; 
\eqe 
recall the discussion just after Definition~\ref{def-lqg-surface}. We have the following description of the conditional law of $\mcl S_\gamma^{J_\gamma}$ given $\mcl A_\gamma$, which we will prove using~\cite[Proposition 5.4]{acsw-loop}.

\begin{lem} \label{lem-weighting-subcritical}
	For each $\gamma \in (\sqrt{8/3},2)$, a.s.\ the conditional law of the origin-containing LQG surface $\mcl S_\gamma^{J_\gamma}$ given 
	$\mcl A_\gamma$ is equal to the law of a $\gamma$-LQG disk with boundary length $\nu_\gamma(\ell_\gamma^{J_\gamma})$ weighted by $\mu_\gamma(U_\gamma^{J_\gamma}) / \mu_\gamma(\BB D)$ (normalized to be a probability measure). 
\end{lem}

We note that $\nu_\gamma(\ell_\gamma^{J_\gamma})$ is a.s.\ determined by $\mcl A_\gamma$, since it is the same as the LQG length of the inner (i.e., non-outer) boundary of $\mcl A_\gamma$, and LQG lengths are intrinsic to LQG surfaces (do not depend on the choice of embedding) as explained at the end of Section~\ref{sec-lqg-surfaces}. Furthermore, the total LQG mass $\mu_\gamma(U_\gamma^{J_\gamma})$ is determined by $\mcl S_\gamma^{J_\gamma}$, and the total LQG area $\mu_\gamma(\BB D) = \mu_\gamma(U_\gamma^{J_\gamma}) + \mu_\gamma(A_\gamma)$ is determined by $\mcl S_\gamma^{J_\gamma}$ and $\mcl A_\gamma$. Hence the conditional law in the lemma statement makes sense.

\begin{proof}[Proof of Lemma~\ref{lem-weighting-subcritical}]
	Define the re-weighted probability measure
	\eqbn
	\BB P^*_\gamma := \frac{\mu_\gamma(\BB D)}{\BB E[\mu_\gamma(\BB D)]} \BB P .
	\eqen
	Note that this makes sense since $\BB E[\mu_\gamma(\BB D)]$ is finite for $\gamma  < 2$, see Remark~\ref{remark-moment}. 
	
	The definition of unmarked $\gamma$-LQG disk with boundary length $\eps$ in \cite{acsw-loop} agrees with that of  Definition~\ref{def-lqg-disk} (see discussion after Definition~\ref{def-lqg-disk}.) The \emph{$\gamma$-LQG disk with a marked interior point and boundary length $\eps$} in \cite{acsw-loop} arises from the unmarked variant via re-weighting by $\mu_\gamma(\BB D)/\BB E[\mu_\gamma(\BB D)]$ and  sampling a point from $\mu_\gamma/\mu_\gamma(\BB D)$, hence its law is precisely $\BB P_\gamma^*$. Thus, by~\cite[Proposition 5.4]{acsw-loop} (which is itself a consequence of \cite[Theorem 1.1]{msw-simple-cle-lqg}), we have the following analog of the lemma statement under $\BB P_\gamma^*$. 
	For each $\gamma \in (\sqrt{8/3},2)$, the $\BB P_\gamma^*$-conditional law of $\mcl S_\gamma^{J_\gamma}$ given 
	$\mcl A_\gamma$ is equal to the law of a $\gamma$-LQG disk with boundary length $\nu_\gamma(\ell_\gamma^{J_\gamma})$ weighted by $\mu_\gamma(U_\gamma^{J_\gamma}) $ (normalized to be a probability measure). Weighting by $\mu_\gamma(\BB D)^{-1}$, we conclude that the above statement holds when $\BB P_\gamma^*$ is replaced by $\BB P$ and $\mu_\gamma(U_\gamma^{J_\gamma})$ is replaced by $\mu_\gamma(U_\gamma^{J_\gamma})/\mu_\gamma (\BB D)$, as needed. 
\end{proof}

To complement Lemma~\ref{lem-disk-conv}, we have the following CLE convergence statement.

\begin{lem} \label{lem-cle-conv}
	As $\gamma\to 2^-$, the origin-surroundings loops and the CLE gaskets satisfy $(\ell_\gamma^{J_\gamma} , \ol G_\gamma) \to (\ell_2^{J_2}, \ol G_2)$ in distribution with respect to the Hausdorff distance on each coordinate.
\end{lem}
\begin{proof}
	This follows from the Brownian loop soup description of CLE~\cite{shef-werner-cle}, see~\cite[Appendix A]{acsw-loop} for details. 
\end{proof}

By combining the LQG convergence statement of Lemma~\ref{lem-disk-conv} and the CLE convergence statement of Lemma~\ref{lem-cle-conv}, we get the following convergence statement for LQG surfaces.

\begin{lem} \label{lem-in-out-conv}
	Let $\mcl S_\gamma^{J_\gamma}$ and $\mcl A_\gamma$ be the LQG surfaces parametrized by the regions inside and outside of the origin-containing loop $ \ell_\gamma^{J_\gamma}$, as in~\eqref{eqn-loop-surface} and~\eqref{eqn-outer-annulus}.   
	By a slight abuse of notation, we identify $\mcl S_\gamma^{J_\gamma}$ with the LQG surface with the annulus topology
	\eqb \label{eqn-puncture}
	\left( U_\gamma^{J_\gamma} \setminus \{0\} , \ell_\gamma^{J_\gamma}  , \Phi|_{U_\gamma^{J_\gamma}}  \right) / {\sim}_{Q_\gamma} . 
	\eqe  
	As $\gamma\to 2^-$, we have $(\mcl S_\gamma^{J_\gamma}  ,\mcl A_\gamma) \to ( \mcl S_2^{J_2} , \mcl A_2)$ in distribution with respect to the topology of Section~\ref{sec-annulus-metric}.   
\end{lem}
\begin{proof}
	For $\gamma\in (0,2]$, let $f_\gamma : \BB D \to  U_\gamma^{J_\gamma}$ be the conformal map such that $f_\gamma(0) =0$ and $f_\gamma'(0) > 0$. Also let $r_\gamma \in (0,1)$ be the conformal modulus of $A_\gamma$ and let $g_\gamma : \BB D \setminus \ol{B_{r_\gamma}(0)} \to A_\gamma$ be the conformal map such that $g_\gamma(1) =1 \in \bdy\BB D$. 
	
	By Lemmas~\ref{lem-disk-conv} and~\ref{lem-cle-conv} and the Skorokhod theorem, we can couple $\{(\Phi_\gamma , \Gamma_\gamma)\}_{\gamma \in (\sqrt{8/3},2]}$ so that a.s.\ $\ep^{-1} \mu_\gamma \to \mu_2$ in the Prokhorov distance and $\ell_\gamma^{J_\gamma} \to \ell_2^{J_2}$ and $\ol G_\gamma \to \ol G_2$ in the Hausdorff distance. 
	In any such coupling, one has a.s.\ $f_\gamma \to f_2$ uniformly on compact subsets of $\BB D$ and a.s.\ $r_\gamma \to r_2$ and $g_\gamma \to g_2$ uniformly on compact subsets of $\BB D\setminus \ol{B_{r_2}(0)}$.  
	This implies the a.s.\ weak convergence of the pullback measures 
	\eqbn
	\ep^{-1} f_\gamma^* \mu_\gamma  \to f_2^* \mu_2 \quad \text{and} \quad \ep^{-1} g_\gamma^* \mu_\gamma \to g_2^* \mu_2  
	\eqen
	and also the a.s.\ Hausdorff distance convergence $f_\gamma^{-1} (\ol G_\gamma)  \to f_2^{-1}(\ol G_\gamma)  $.
	This, in turn, implies the lemma statement by the definition of the topology on LQG surfaces from~\eqref{eqn-annulus-metric}. 
\end{proof}

We will now combine Lemmas~\ref{lem-weighting-subcritical} and~\ref{lem-in-out-conv} to deduce Lemma~\ref{lem-weighting}, which is the analog of Lemma~\ref{lem-weighting-subcritical} for $\gamma =2$. To that end, we will need the following elementary fact about convergence of conditional distributions, see, e.g.,~\cite[Lemma 4.3]{gp-sle-bubbles}.
\begin{lem}\label{lem-cond-limit}
	Let $(X_n, Y_n)$ be a sequence of pairs of random variables taking values in a product of separable metric spaces $\Omega_X \times \Omega_Y$ and let $(X,Y)$ be another such pair such that $(X_n , Y_n) \to (X, Y)$ in law. Suppose further that there is a family of probability measures $\{ P_y:y \in \Omega_Y\}$ on $\Omega_X$, indexed by $\Omega_Y$, such that for every bounded continuous function $f: \Omega_X \to \BB R$, 
	\[(\BB E[f(X_n)\mid Y_n], Y_n) \to (\BB E_{P_Y}(f), Y) \quad \text{ in law}.\]
	Then $P_Y$ is the regular conditional law of $X$ given $Y$. 
\end{lem}

\begin{lem} \label{lem-weighting}
	Almost surely, the conditional law of the origin-containing LQG surface $\mcl S_2^{J_2}$ given 
	$\mcl A_2$ is equal to the law of a critical LQG disk with boundary length $\nu_2(\ell_2^{J_2})$ weighted by $\mu_2(U_2^{J_2}) / \mu_2(\BB D)$ (normalized to be a probability measure). 
	Furthermore, we have the joint convergence in law
	\eqb \label{eqn-origin-loop-conv}
	\left( \mcl S_\gamma^{J_\gamma}  ,\mcl A_\gamma , \ep^{-1} \nu_\gamma(\ell_\gamma^{J_\gamma}) \right) 
	\to \left( \mcl S_2^{J_2} , \mcl A_2 , \nu_2(\ell_2^{J_2}) \right) 
	\eqe  
	and this convergence in law also holds jointly with the convergence in law from Lemmas~\ref{lem-disk-conv} and~\ref{lem-cle-conv}.
\end{lem}
\begin{proof}
	By Lemma~\ref{lem-weighting-subcritical}, the lemma statement is true with $\gamma=2$ replaced by any $\gamma\in (\sqrt{8/3},2)$. 
	
	To pass to the limit as $\gamma\to 2$, we first address the convergence of the re-scaled LQG length $\ep^{-1} \nu_\gamma(\ell_\gamma^{J_\gamma})$. 
	By Lemmas~\ref{lem-in-out-conv} and~\ref{lem-length-tight}, for any sequence of $\gamma$-values increasing to 2, we can find a subsequence $\mcl G$ and a coupling $(\mcl S_2^{J_2} , \mcl A_2)$ with a random variable $X$ such that along $\mcl G$,
	\eqb  \label{eqn-weighting-conv}
	\left( \mcl S_\gamma^{J_\gamma}  ,\mcl A_\gamma , \ep^{-1} \nu_\gamma(\ell_\gamma^{J_\gamma}) \right) 
	\to \left( \mcl S_2^{J_2} , \mcl A_2 , X \right) 
	\eqe 
	in distribution. By the Skorokhod theorem, we can couple so that the convergence~\eqref{eqn-weighting-conv} in fact occurs a.s.\ along $\mcl G$.  
	
	By Lemma~\ref{lem-disk-conv}, the a.s.\ convergence~\eqref{eqn-weighting-conv} implies that a.s.\ the conditional law of $\mcl S_\gamma^{J_\gamma}$ given $\mcl A_\gamma$ converges weakly as $\gamma \to 2^-$ to the law of a critical LQG disk with boundary length $X$ weighted by $\mu_2(U_2^{J_2}) / \mu_2(\BB D)$ (normalized to be a probability measure). 
	By Lemma~\ref{lem-cond-limit},
	this implies that the conditional law of $\mcl S_2^{J_2}$ given $\mcl A_2$ is that of a critical LQG disk with boundary length $X$ weighted by $\mu_2(U_2^{J_2}) / \mu_2(\BB D)$ (normalized to be a probability measure). 
	This, in turn, implies that we must in fact have $X = \nu_2(\ell_2^{J_2})$. This gives the lemma statement and also the convergence in law~\eqref{eqn-origin-loop-conv}. 
	
	To get that the convergence in~\eqref{eqn-origin-loop-conv} also holds jointly with the convergence in law from Lemmas~\ref{lem-disk-conv} and~\ref{lem-cle-conv}, we simply choose our initial coupling so that the convergence statements from those lemmas occur a.s.\ as well.  
\end{proof}

To deduce Theorem~\ref{thm-critical-cle4} from Lemma~\ref{lem-weighting}, we need the following elementary probabilistic lemma. 

\begin{lem} \label{lem-cond-weight}
	Let $\{X_j\}_{j\in\BB N}$ be random variables taking values in a separable metric space $\mcl X$.
	Let $f : \mcl X \to [0,\infty)$ be a measurable function such that $\sum_{j=1}^\infty f(X_j) < \infty$ a.s.
	Conditional on $\{X_j\}_{j\in\BB N}$, let $J \in \BB N$ be sampled from the probability measure 
	\eqbn
	\BB P\left[ J = k \,|\, \{X_j\}_{j\in\BB N} \right]  = \frac{f(X_k)}{\sum_{j=1}^\infty f(X_j)} .
	\eqen
	Then for each $k\in\BB N$, the conditional law of $X_k$ given $\{J=k\}$ and $\{X_j\}_{j\not=k}$ is equal to the conditional law of $X_k$ given $\{X_j\}_{j\not=k}$ weighted by the Radon-Nikodym derivative 
	\eqbn
	\frac{f(X_k)}{\sum_{j=1}^\infty f(X_j)} \times \BB E\left[ \frac{f(X_k)}{\sum_{j=1}^\infty f(X_j)} \,|\, \{X_j\}_{j\not=k} \right]^{-1} .
	\eqen
\end{lem}
\begin{proof}
	To lighten notation, we introduce the conditional probability measure 
	\eqbn
	\ol{\BB P}_k[\cdot] = \BB P[\cdot \,|\, \{X_j\}_{j\not=k} ] 
	\eqen
	and we write $\ol{\BB E}_k$ for the corresponding expectation.
	
	Let $A\subset \mcl X$ be measurable. By the definition of $J$,  
	\eqb \label{eqn-cond-weight-top}
	\ol{\BB P}_k\left[ J = k \,|\, X_k \in A \right] 
	=  \frac{1}{\ol{\BB P}_k[X_k\in A]} \ol{\BB E}_k\left[ \frac{f(X_k)}{\sum_{j=1}^\infty f(X_j)} \BB 1_{(X_k \in A)} \right] 
	\eqe  
	and
	\eqb  \label{eqn-cond-weight-bottom}
	\ol{\BB P}_k\left[ J = k  \right] 
	=  \ol{\BB E}_k\left[ \frac{f(X_k)}{\sum_{j=1}^\infty f(X_j)}   \right] . 
	\eqe  
	By Bayes' rule, 
	\eqbn
	\ol{\BB P}_k\left[ X_k  \in A \,|\, J = k\right] 
	= \frac{\ol{\BB P}_k\left[ J = k \,|\, X_k \in A \right]}{\ol{\BB P}_k\left[ J = k \right]} \ol{\BB P}_k\left[   X_k \in A \right] .
	\eqen
	Plugging in~\eqref{eqn-cond-weight-top} and~\eqref{eqn-cond-weight-bottom} concludes the proof.
\end{proof}

\begin{proof}[Proof of Theorem~\ref{thm-critical-cle4}]
	Recall the LQG surfaces $\mcl S_2^j$ parametrized by the interiors of the CLE$_4$ loops in~\eqref{eqn-loop-surface} and the random variable $J_2$ which tells us the index of the origin-containing loop. 
	For a given $k\in\BB N$, we will give two different descriptions of the conditional law of $\mcl S_2^{k}$ given $\{ \nu_2(\ell_2^j)\}_{j\in\BB N}$ and $\{\mcl S_2^j \}_{ j \not= k}$, one coming from Lemma~\ref{lem-cond-weight} and the other coming from Lemma~\ref{lem-weighting}. We will then equate these two descriptions to conclude the proof. 
	
	We apply Lemma~\ref{lem-cond-weight} under the conditional law given $\{ \nu_2(\ell_2^j)\}_{j\in\BB N}$, with the random variables $X_j = \mcl S_2^j$ as in~\eqref{eqn-loop-surface} and the function $f(\mcl S_2^j) = \mu_2(U_2^j)$. 
	Due to~\eqref{eqn-marked-loop-law}, Lemma~\ref{lem-cond-weight} in this setting tells us that for each $k\in\BB N$, the conditional law of $\mcl S_2^k$ given $\{J_\gamma = k\}$, $\{ \nu_2(\ell_2^j)\}_{j\in\BB N}$, and $\{\mcl S_2^j \}_{ j \not= k}$ is the same as the conditional law of $\mcl S_2^{k}$ given $\{ \nu_2(\ell_2^j)\}_{j\in\BB N}$ and $\{\mcl S_2^j \}_{ j \not= k}$ weighted by 
	\eqb  \label{eqn-use-cond-weight}
	\frac{\mu_2(U_2^k)}{\mu_2(\BB D)}
	\eqe 
	(normalized to be a probability measure). 
	
	On the other hand, recall the outer LQG surface $\mcl A_2$ of~\eqref{eqn-outer-annulus}. The conditional law of $\mathcal S_2^{J_2}$ given $\mathcal A_2$ (Lemma~\ref{lem-weighting}) depends only on $\nu_2(\ell_2^{J_2})$ and $\mu_2(\mathcal A_2 )$, which are determined by $J_2$, $\{ \nu_2(\ell_2^j)\}_{j\in\BB N}$, and $\{\mcl S_2^j \}_{ j \not= J_2}$.
	Hence Lemma~\ref{lem-weighting} implies that the conditional law of $\mcl S_2^{J_2}$ given $J_2$, $\{ \nu_2(\ell_2^j)\}_{j\in\BB N}$, and $\{\mcl S_2^j \}_{ j \not= J_2}$ is as described in Lemma~\ref{lem-weighting}. Equivalently, for each $k\in\BB N$, the conditional law of $\mcl S_2^k$ given $\{J_\gamma = k\}$, $\{ \nu_2(\ell_2^j)\}_{j\in\BB N}$, and $\{\mcl S_2^j \}_{ j \not= k}$ is the same as the law of a critical LQG disk with boundary length $\nu_2(\ell_2^k)$, weighted by
	\eqb  \label{eqn-use-weighting}
	\frac{\mu_2(U_2^k)}{\mu_2(\BB D)}
	\eqe 
	(normalized to be a probability measure). 
	
	By comparing the descriptions of the conditional law in the sentences containing~\eqref{eqn-use-cond-weight} and~\eqref{eqn-use-weighting}, we see that the conditional law of $\mcl S_2^k$ given only $\{ \nu_2(\ell_2^j)\}_{j\in\BB N}$ and $\{\mcl S_2^j \}_{ j \not= k}$ is that of a critical LQG disk with the given boundary length $\nu_2(\ell_2^k)$. This holds for any $k\in\BB N$, so we get that the LQG surfaces $\{\mcl S_2^j\}_{j\in\BB N}$ are conditionally independent LQG disks given their boundary lengths. 
\end{proof}

\subsection{Proof of Theorem~\ref{thm-critical-lengths}}
\label{sec-length-proof}

We continue to use the notation and setup of Section~\ref{sec-details}. In particular, $\Phi_\gamma$ is an embedding of a $\gamma$-LQG disk into $\BB D$, $\Gamma_\gamma$ is an independent CLE$_{\gamma^2}$, $\{\ell_\gamma^j\}_{j\in\BB N}$ denote the loops of $\Gamma_\gamma$ enumerated in decreasing order of their $\nu_\gamma$-lengths, and $\ep = 2(2-\gamma)$.

Just as in the case of Theorem~\ref{thm-critical-cle4}, we will deduce Theorem~\ref{thm-critical-lengths} from its subcritical analog and a limiting argument. Here is the subcritical statement we will use. We follow the presentation of \cite[Proposition 4.1]{acsw-loop}. 

\begin{thm}[{\cite{msw-simple-cle-lqg,ccm-cascade,bbck-growth-frag}}] \label{thm-subcritical-lengths}
	Let $\gamma\in (\sqrt{8/3},4)$ and let $\beta = \beta(\gamma) = \frac{4}{\gamma^2}+\frac12 \in (\frac32,2)$. 
	Let $\zeta_\gamma$ be a $\beta$-stable L\'evy process with no downward jumps, with L\'evy measure $\BB 1_{(x>0)} x^{-\beta-1} \,dx$.
	Let $P_\gamma$ be the law of $\zeta_\gamma$.  
	Let $\tau_\gamma  = \inf\{t\geq 0 : \zeta_\gamma(t)  = -1\}$.
	Let $\{X_\gamma^j\}_{j\in\BB N}$ be the sequence of upward jumps of $\zeta_\gamma$ on $[0,\tau_\gamma]$, listed in decreasing order.
	Then the law of the re-scaled $\gamma$-LQG lengths of the CLE$_{\gamma^2}$ loops $\{\ep^{-1} \nu_\gamma(\ell_\gamma^j)\}_{j\in\BB N}$ (which are listed in decreasing order, see~\eqref{eqn-ordered-loops}) is the same as the law of $\{X_\gamma^j\}_{j\in\BB N}$ under the re-weighted law $(\tau_\gamma^{-1} / \BB E[\tau_\gamma^{-1}]) P_\gamma$. 
\end{thm}
As explained at the end of~\cite[Section 1]{ccm-cascade}, Theorem~\ref{thm-subcritical-lengths} follows by combining results in~\cite{msw-simple-cle-lqg,ccm-cascade,bbck-growth-frag}. Further details are given in \cite[Section 4.3]{acsw-loop}.

Note that $(\BB D,\Phi_\gamma)$ is an LQG disk of boundary length $\ep$, rather than boundary length 1, which is the reason for the re-scaling by $\ep^{-1}$.

Theorem~\ref{thm-subcritical-lengths} directly implies the following.

\begin{lem} \label{lem-levy-conv}
	As $\gamma\to 2^-$, the joint law of the re-scaled loop lengths $\{\ep^{-1} \nu_\gamma(\ell_\gamma^j)\}_{j\in\BB N}$ converges to the joint law of $\{X^j\}_{j\in\BB N}$, where $\{X^j\}_{j\in\BB N}$ is sampled from the weighted law $(\tau^{-1} / \BB E[\tau^{-1}]) P$ as in Theorem~\ref{thm-critical-lengths}. 
\end{lem}
\begin{proof}
	This follows from Theorem~\ref{thm-subcritical-lengths} and standard convergence results for L\'evy processes (see, e.g.,~\cite[Theorem 13.17]{kallenberg}).  
\end{proof}

To prove Theorem~\ref{thm-critical-lengths}, it remains to show that $\{\ep^{-1} \nu_\gamma(\ell_\gamma^j)\}_{j\in\BB N}$ also converges in law to $\{\nu_2(\ell_2^j)\}_{j\in\BB N}$. We will prove this using the convergence statement in Lemma~\ref{lem-weighting}. A key input is the following tightness statement. 

\begin{lem} \label{lem-index-tight}
	The random variables $\{J_\gamma\}_{\gamma \in (\sqrt{8/3},2)}$ as defined in~\eqref{eqn-marked-loop-def} are tight. 
\end{lem}
\begin{proof}
	Fix $\delta >0$. For $\gamma \in (\sqrt{8/3},2)$, let $N_\gamma$ be the number of loops of $\Gamma_\gamma$ with $\nu_\gamma$-length at least $\ep \delta$. We first use a similar argument to the proof of Lemma~\ref{lem-length-tight} to show that the random variable $N_\gamma$ is tight. 
	
	By the definition of the $\ell_\gamma^j$s, we have $\nu_\gamma(\ell_\gamma^j) \geq \ep\delta$ for each $j=1,\dots,N_\gamma$. Let $\mcl F_\gamma$ be the $\sigma$-algebra generated by $\{\nu_\gamma(\ell_\gamma^j)\}_{j\in\BB N}$. By Theorem~\ref{thm-subcritical-coupling}, under the conditional law given $\mcl F_\gamma$, the random variables $\ep^2 \mu_\gamma(U_\gamma^j) / [\nu_\gamma(\ell_\gamma^j)]^2$ for $j\in\BB N$ are i.i.d.\ and each has the law of the $\gamma$-LQG area of a $\gamma$-LQG disk of boundary length $\ep$. By Lemma~\ref{lem-disk-conv}, this LQG area scaled by $\ep^{-1}$ converges in distribution to an a.s.\ positive random variable. Hence there is a universal constant $c > 0$ such that for each $j \in \BB N$, 
	\eqbn
	\BB P\left[ \mu_\gamma(U_\gamma^j)  \geq c \ep^{-1} [\nu_\gamma(\ell_\gamma^j)]^2 \,|\, \mcl F_\gamma\right] \geq \frac34 .
	\eqen
	By comparison to a binomial random variable, we obtain
	\eqbn
	\BB P\left[ \#\left\{j \in [1,N_\gamma]\cap\BB Z  : \mu_\gamma(U_\gamma^j)  \geq c \ep \delta^2 \right\} \geq \frac12 N_\gamma \,|\, \mcl F_\gamma\right] \geq 1 - 2 e^{-C N_\gamma} 
	\eqen
	where $C  > 0$ is a universal constant. Since $\mu_\gamma(\BB D) = \sum_{j=1}^\infty \mu_\gamma(U_\gamma^j)$, we get
	\eqb \label{eqn-big-total-mass}
	\BB P\left[   \mu_\gamma(\BB D)  \geq \frac12 c \ep \delta^2  N_\gamma  \,|\, \mcl F_\gamma\right] \geq 1 - 2 e^{-C N_\gamma} . 
	\eqe 
	Since $\ep^{-1} \mu_\gamma(\BB D)$ is a tight random variable by Lemma~\ref{lem-disk-conv}, we deduce from~\eqref{eqn-big-total-mass} that $N_\gamma$ is also a tight random variable.
	
	By Lemma~\ref{lem-weighting}, we see that $\ep^{-1} \nu_\gamma(\ell_\gamma^{J_\gamma})$ converges in distribution to the a.s.\ positive random variable $\nu_2(\ell_2^{J_2})$. Hence, when $\delta > 0$ is small it holds with high probability (uniformly in $\gamma$) that $\nu_\gamma(\ell_\gamma^{J_\gamma}) \geq \ep\delta$, and hence $J_\gamma \leq N_\gamma$. So, the tightness of $N_\gamma$ implies the tightness of $J_\gamma$.
\end{proof}

In order to deduce the convergence of the loop lengths $\{\nu_\gamma( \ell_\gamma^j)\}_{j\in\BB N}$ from the convergence of the single loop length $\nu_\gamma(\ell^{J_\gamma})$, we consider countably many copies of $J_\gamma$. 
For $\gamma \in (\sqrt{8/3},2]$, condition on $(\Phi_\gamma,\Gamma_\gamma)$ and let $\{Z_\gamma(m)\}_{m\in\BB N}$ be conditionally i.i.d.\ samples from $\mu_\gamma$, normalized to be a probability measure. For $m\in\BB N$, let $J_\gamma(m)$ be chosen so that 
\eqb \label{eqn-infinite-pts-def} 
Z_\gamma(m) \in U_\gamma^{J_\gamma(m)} . 
\eqe 
Note that $(\Phi_\gamma,\Gamma_\gamma,J_\gamma(m)) \eqD (\Phi_\gamma,\Gamma_\gamma,J_\gamma)$ for each $m\in\BB N$. 

\begin{lem} \label{lem-infinite-marked-pts}
	With $J_\gamma(m)$ defined as in~\eqref{eqn-infinite-pts-def}, we have the convergence of joint laws
	\eqb \label{eqn-infinite-marked-pts}
	\{ \ep^{-1} \nu_\gamma(\ell_\gamma^{J_\gamma(m )}) \}_{m\in\BB N}  \to  \{\nu_2( \ell_2^{J_2(m )} ) \}_{m\in\BB N}    .  
	\eqe  
\end{lem}
\begin{proof} 
	By Lemma~\ref{lem-disk-conv}, we have the convergence of joint laws
	\eqbn
	\left( \Phi_\gamma , \ep^{-1} \mu_\gamma , \ep^{-1} \nu_\gamma ,  \{Z_\gamma(m)\}_{m\in\BB N}  \right) 
	\to \left( \Phi_2  ,   \mu_2 ,  \nu_2 ,   \{Z_2(m)\}_{m\in\BB N}  \right)  .
	\eqen
	By combining this with the convergence part of Lemma~\ref{lem-weighting}, we get for each fixed $m_0 \in\BB N$ the convergence of joint laws
	\allb \label{eqn-infinite-marked-loops0}
	&\left( \Phi_\gamma , \ep^{-1} \mu_\gamma , \ep^{-1} \nu_\gamma , \{Z_\gamma(m)\}_{m\in\BB N} , \ell_\gamma^{J_\gamma(m_0)} , \ol G_\gamma  , \ep^{-1} \nu_\gamma(\ell_\gamma^{J_\gamma(m_0)}) \right) \notag\\  
	&\qquad \to \left( \Phi_2  ,  \mu_2 ,   \nu_2 ,    \{Z_2(m)\}_{m\in\BB N}  ,  \ell_2^{J_2(m_0)} , \ol G_2  , \nu_2(\ell_2^{J_2(m_0)}) \right)   
	\alle
	where the fifth and sixth components are equipped with the Hausdorff distance. 
	Since $\ell_2^{J_2(m_0)}$ and $\nu_2(\ell_2^{J_2(m_0)})$ are measurable functions of $\Phi_2$, $\{Z_2(m)\}_{m\in\BB N}$, and the CLE$_4$ gasket $\ol G_2$, we see that the convergence~\eqref{eqn-infinite-marked-loops0} implies that in fact we have the convergence of joint laws 
	\allb \label{eqn-infinite-marked-loops}
	&\left( \Phi_\gamma , \ep^{-1} \mu_\gamma , \ep^{-1} \nu_\gamma , \{Z_\gamma(m)\}_{m\in\BB N} , \{ \ell_\gamma^{J_\gamma(m )} \}_{m\in\BB N} , \ol G_\gamma  ,  \{ \ep^{-1} \nu_\gamma(\ell_\gamma^{J_\gamma(m )}) \}_{m\in\BB N} \right) \notag\\  
	&\qquad \to \left( \Phi_2  ,   \mu_2 ,   \nu_2 ,    \{Z_2(m)\}_{m\in\BB N}  ,  \{\ell_2^{J_2(m )} \}_{m\in\BB N} , \ol G_2  , \{\nu_2( \ell_2^{J_2(m )} ) \}_{m\in\BB N} \right)  .  
	\alle  
	This in particular implies~\eqref{eqn-infinite-marked-pts}.
\end{proof}

For $\gamma\in (\sqrt{8/3},2]$, define the points $Z_\gamma(m)$ and the indices $J_\gamma(m)$ as in~\eqref{eqn-infinite-pts-def}. For $j\in \BB N$, let 
\eqb \label{eqn-inverse-index}
M_\gamma^j := \min\left\{ m \in \BB N : Z_\gamma(m) \in U_\gamma^j\right\} 
= \min\left\{m\in\BB N : J_\gamma(m) = j \right\} .
\eqe

\begin{lem} \label{lem-inverse-tight}
	For each fixed $j\in \BB N$, the laws of the random variables $\{M_\gamma^j\}_{\gamma\in (\sqrt{8/3},2)}$ are tight. 
\end{lem}
\begin{proof}
	Conditional on $(\Phi_\gamma,\Gamma_\gamma)$, the random variables $\{Z_\gamma(m)\}_{m\in\BB N}$ are i.i.d.\ samples from $\mu_\gamma$. Hence, the conditional law of $M_\gamma^j$ given $(\Phi_\gamma,\Gamma_\gamma)$ is geometric with success probability $ \mu_\gamma(U_\gamma^j) / \mu_\gamma(\BB D)$. We know that $\ep^{-1} \mu_\gamma(\BB D)$ is a tight random variable by Lemma~\ref{lem-disk-conv}. So, we just need to show that $[ \ep^{-1} \mu_\gamma(U_\gamma^j) ]^{-1}$ is tight. 
	From Lemma~\ref{lem-levy-conv}, we know that $[\ep^{-1} \nu_\gamma(\ell_\gamma^j)]^{-1}$ is tight. 
	The tightness of $[ \ep^{-1} \mu_\gamma(U_\gamma^j) ]^{-1}$ follows from this and the proof of Lemma~\ref{lem-length-tight} (see~\eqref{eqn-big-area-pos}).\footnote{We could have alternatively deduced Lemma~\ref{lem-length-tight} from Lemma~\ref{lem-levy-conv}, but we chose to give a direct proof of Lemma~\ref{lem-length-tight} in order to make the proof of Theorem~\ref{thm-critical-cle4} more self-contained.}
\end{proof}

The following is the main input in the proof of Theorem~\ref{thm-critical-lengths}. 

\begin{prop} \label{prop-ordered-loop-conv}
	We have the convergence of joint laws
	\eqb
	\{ \ep^{-1} \nu_\gamma(\ell_\gamma^j)\}_{j\in\BB N}  \to \{\nu_2(\ell_2^j)\}_{j\in\BB N} .
	\eqe
\end{prop}
\begin{proof}
	By Lemma~\ref{lem-index-tight}, for each $m\in\BB N$ the random variables $\{J_\gamma(m)  \}_{\gamma\in (\sqrt{8/3},2)}$ are tight. By combining this with Lemmas~\ref{lem-disk-conv} and~\ref{lem-cle-conv}, we find that for each sequence of $\gamma$-values increasing to 2, there exists a subsequence $\mcl G$ and a coupling of $(\Phi_2,\Gamma_2,\{Z_2(m)\}_{m\in\BB N} )$ with a sequence of integer-valued random variables $\{\wt J(m) \}_{m\in\BB N}$ such that the following is true. As $\mcl G\ni\gamma \to 2$, we have the convergence of joint laws
	\eqb \label{eqn-index-conv}
	\{(  \ep^{-1} \nu_\gamma(\ell_\gamma^{J_\gamma(m )}) ,  J_\gamma(m) ) \}_{m\in\BB N} \to \{( \nu_2(\ell_2^{J_2(m)}) ,    \wt J(m) ) \}_{m\in\BB N} .
	\eqe
	By the Skorokhod theorem, we can couple so that the convergence~\eqref{eqn-index-conv} occurs a.s.\ along $\mcl G$. 
	
	Since $J_\gamma(m)$ is an integer-valued random variable, the convergence~\eqref{eqn-index-conv} implies that for each $m\in\BB N$, it is a.s.\ the case that $J_\gamma(m) = \wt J(m)$ for each $\gamma\in\mcl G$ sufficiently close to 2. 
	Hence~\eqref{eqn-index-conv} implies that a.s., for each $m\in\BB N$,  
	\eqb  \label{eqn-length-ssl-conv} 
	\lim_{\mcl G \ni \gamma} \ep^{-1} \nu_\gamma(\ell_\gamma^{\wt J (m )}) = \nu_2( \ell_2^{\wt J ( m)} ) .
	\eqe 
	
	By~\eqref{eqn-length-ssl-conv}, to conclude the proof it suffices to show that for each fixed $j \in \BB N$, a.s.\ there exists $m\in\BB N$ such that $\wt J(m) = j$. To this end, we use Lemma~\ref{lem-inverse-tight} to get that a.s.\ $M_\gamma^{j }$ does not diverge to $\infty$ as $\mcl G \ni \gamma \to 2^-$. 
	Hence there exists $m\in\BB N$ and a subsequence $\mcl G' \subset\mcl G$ (both random) such that $M_\gamma^{j } = m$ for each $\gamma \in \mcl G'$. Then for $\gamma \in \mcl G'$ sufficiently close to 2, we have $j  = J_\gamma(m)= \wt J(m)$. 
\end{proof}

\begin{proof}[Proof of Theorem~\ref{thm-critical-lengths}]
	This follows by combining Lemma~\ref{lem-levy-conv} and Proposition~\ref{prop-ordered-loop-conv}. 
\end{proof}

\bibliography{cibib, bib}
\bibliographystyle{hmralphaabbrv}

\end{document}